\documentclass[a4paper,reqno]{amsart}
\usepackage[british]{babel}
\usepackage{mathptmx}
\usepackage{courier}
\usepackage{amssymb}
\usepackage{mathtools}
\usepackage{amsthm}
\usepackage{enumerate}
\usepackage{tikz}
\usepackage{nicefrac}

\newtheorem{theorem}{Theorem}
\newtheorem{proposition}[theorem]{Proposition}

\newtheorem*{remark*}{Remark}

\newcommand{\nats}{\mathbb{N}}
\newcommand{\natswith}{\nats_{0}}

\newcommand{\values}{\mathcal{X}}

\newcommand{\lexp}{\underline{E}}
\newcommand{\lexpW}{\lexp_\mathrm{W}}
\newcommand{\lexpV}{\lexp_\mathrm{V}}

\newcommand{\gambles}{\mathcal{L}} 

\newcommand{\paths}{\Omega}
\newcommand{\path}{\omega} 

\newcommand{\proc}{\mathcal{F}}
\newcommand{\sel}{\mathcal{S}}
\newcommand{\sels}{\mathbb{S}}

\DeclareMathOperator{\cyl}{cyl}

\newcommand{\init}{\square}

\newcommand{\precedes}{\sqsubseteq} 
\newcommand{\sprecedes}{\sqsubset}

\newcommand{\reals}{\mathbb{R}}

\newcommand{\martin}{\mathcal{M}}
\newcommand{\martins}{\underline{\mathbb{M}}}
\newcommand{\martinsb}{\martins_{\mathrm{b}}}

\begin{document}
\title[Continuity of imprecise stochastic processes for monotone sequences]{Continuity of imprecise stochastic processes\\ with respect to the pointwise convergence\\ of monotone sequences}
\author{Jasper De Bock}
\address{Ghent University, IDLab,
  Technologiepark -- Zwijnaarde 914, 9052 Zwijnaarde, Belgium}
\email{jasper.debock@ugent.be}
\author{Gert de Cooman}

\begin{abstract}
 We consider the joint lower expectation of a finite-state imprecise stochastic process, defined using either the Ville-Vovk-Shafer natural extension or the Williams natural extension. In both cases, we show that it is continuous with respect to the pointwise convergence of non-decreasing sequences of real-valued functions $f_n$, $n\in\natswith$, where each $f_n$ is $n$-measurable. For the Ville-Vovk-Shafer natural extension, a similar result is shown to hold for non-increasing sequences, provided that they converge to a bounded function.
\end{abstract}

\maketitle

\section{Preliminaries}
\label{sec:prelim}

Consider an infinite sequence $X_1\dots,X_n,\dots$ of states, where, for every time point $n\in\nats$,\footnote{We use $\nats$ to denote the set of all natural numbers (without zero); if zero is included, we write $\natswith\coloneqq\nats\cup\{0\}$.} the state $X_n$ takes values $x_n$ in a finite set $\values_n$. A real-valued function on $\values_n$ is called a gamble on $\values_n$. We use $\gambles(\values_n)$ to denote the set of all such gambles. A finite initial sequence of states $x_1 \ldots x_n\in\times_{i=1}^n\values_i$ is called a situation. We also allow for $n=0$, which corresponds to the (abstract) initial situation $\init$. Generic situations are denoted by $s$, $s'$ or $u$. For any situation $s$, we denote its length by $\ell(s)$. For example: $\ell(\init)=0$ and, for $s=x_1\ldots x_n$, $\ell(s)=n$. If $s'$ is an initial segment of $s$, in the sense that there are $k,n\in\natswith$ such that $k\leq n$, $s=x_1\dots x_n\in\values^n$ and $s'=x_1\dots x_k$, we write $s'\precedes s$. If $s'\precedes s$ and $s'\neq s$, we write $s'\sprecedes s$.

For every situation $s$, we are given a local lower expectation functional---a lower envelope of expectation operators or, equivalently, a superlinear functional that dominates the infimum---$\underline{Q}(\cdot\vert s)$ on $\gambles(\values_{\ell(s)+1})$. For every gamble $f$ on $\values_{\ell(s)+1}$, the corresponding lower expectation is denoted by $\underline{Q}(f\vert s)$.



A path $\path$ is an infinite sequence of states $x_1 \ldots x_n \ldots \in\times_{i\in\nats}\values_i$. We take the possibility space $\paths$ to be the set of all paths. 
For every $\path=x_1\ldots x_n\ldots\in\paths$ and every $n\in\natswith$, we also consider the situation $\path_n\coloneqq x_1\ldots x_n$, with $\path_0=\init$ as a special case. Furthermore, for every situation $s$, the cylinder set $\smash{\cyl(s)\coloneqq\{\path\in\paths\colon\path_{\ell(s)}=s\}}$ is the set of all paths that have $s$ as their initial segment. Note that $\cyl(\init)=\paths$. An extended real-valued\footnote{This means that it takes values in $\reals\cup\{-\infty,+\infty\}$.} function $f$ on $\paths$ is said to be $n$-measurable if $f(\path)=f(\path')$ for all $\path,\path'\in\paths$ such that $\path_n=\path'_n$ or, equivalently, if for all situations $s\in\values^n$, $f$ is constant on $\cyl(s)$.

A process $\proc$ is a function defined on all situations. If $\proc$ is extended real-valued, we can associate with it an extended real-valued function $\limsup\proc$ on $\paths$, defined for all $\path\in\paths$ by
\vspace{-7pt}
\begin{align}\label{eq:deflimsup}
\limsup\proc(\path)
\coloneqq
\limsup_{n\rightarrow+\infty}\proc(\path_n)
=\lim_{n\rightarrow+\infty}\sup_{m\geq n}\proc(\path_m)=\inf_{n\in\natswith}\sup_{m\geq n}\proc(\path_m)
.
\end{align}
A process that maps each situation $s$ to a gamble on $\smash{\values_{\ell(s)+1}}$ is called a selection, and will be denoted by $\sel$. With every such selection, we associate a real-valued process $\proc^\sel$, defined recursively by
\begin{equation*}
\proc^\sel(\init)\coloneqq0
\text{ and }
\proc^\sel(sx)\coloneqq\proc^\sel(s)+\sel(s)(x)
\text{ for all situations $s$ and all $x\in\values_{\ell(s)+1}$.}
\end{equation*}
Hence, for every $\path=x_1\ldots x_n\ldots\in\paths$ and $n\in\natswith$, we find that
\begin{equation}\label{eq:Fnclosedform}
\proc^\sel(\path_n)=\sum_{i=1}^{n}\sel(\path_{i-1})(x_i).
\end{equation}
Conversely, with any real-valued process $\proc$, we associate a selection $\Delta\proc$, called the difference process, defined by
\vspace{4pt}
\begin{equation*}
\Delta\proc(s)(x)\coloneqq\proc(sx)-\proc(s)
\text{~for every situtation $s$ and any $x\in\values_{\ell(s)+1}$,}
\vspace{5pt}
\end{equation*}
for which we know that $\martin=\martin(\init)+\proc^{\Delta\martin}$. A selection $\sel$ is called almost-desirable when
\vspace{-4pt}
\begin{equation*}
\underline{Q}(\sel(s)\vert s)\geq0
\text{ for all situations $s$.}
\vspace{4pt}
\end{equation*}
The set of all almost-desirable selections is denoted by $\sels$. 
A real process $\martin$ for which  $\Delta\martin\in\sels$ is called a submartingale. The set of all submartingales is denoted by $\martins$.
A submartingale $\martin$ is said to be bounded above if there is some $B\in\reals$ such that, for every situation $s$, $\martin(s)\leq B$. The set of all submartingales that are bounded above is denoted by $\martinsb$.

\section{Joint lower expectations}

We now have all the terminology needed to define the joint lower expectations that we are interested in. For every extended real-valued function $f$ on $\paths$ and every situation $u$, the Williams natural extension is given by
\begin{equation}\label{eq:Williams}
\lexpW(f\vert u)\coloneqq
\sup\{
\martin(u)
\colon 
\martin\in\martinsb, n\in\natswith \text{ and }
f(\path)\geq\martin(\path_n) 
\text{ for all $\path\in\cyl(u)$}
\}
\end{equation}
and the Ville-Vovk-Shafer natural extension~\cite{cooman2015:markovergodic} is given by
\begin{equation}\label{eq:Ville}
\lexpV(f\vert u)\coloneqq
\sup\{
\martin(u)
\colon 
\martin\in\martinsb \text{ and }
f(\path)\geq\limsup\martin(\path) 
\text{ for all $\path\in\cyl(u)$}
\}.
\end{equation}
Of these two natural extensions, the Williams natural extension is the more conservative one.
\begin{proposition}\label{prop:WversusV}
For any situation $u$ and any extended real-valued function $f$ on $\paths$, we have that $\lexpV(f\vert u)\geq\lexpW(f\vert u)$.
\end{proposition}
\begin{proof}
Because of to Equations~\eqref{eq:Williams} and~\eqref{eq:Ville}, it clearly suffices to show that for any submartingale $\martin\in\martinsb$ and any $n\in\natswith$, there is a submartingale $\martin^{*}\in\martinsb$ for which $\martin(\path_n)=\limsup\martin^{*}(\path)$ for all $\path\in\paths$. So consider any $\martin\in\martinsb$ and any $n\in\natswith$. 
We let $\martin^{*}$ be the unique real proces such that $\martin^{*}(\init)=\martin(\init)$ and
\begin{equation}\label{eq:selstar1}
\Delta\martin^{*}(s)
\coloneqq
\begin{cases}
\Delta\martin(s) & \text{if $\ell(s)<n$}\\
0 & \text{otherwise}
\end{cases}
\text{~~~for all situations $s$.}
\end{equation}
Since $\martin^*$ is clearly bounded from above and $\Delta\martin^{*}$ is clearly almost-desirable, we have that $\martin^*\in\martinsb$.
Consider now any $\path=x_1\ldots x_n\ldots\in\paths$. Then for all $k\geq n$, we have that
\begin{equation*}
\martin^{*}(\path_k)
=\martin^{*}(\init)+\sum_{i=1}^{k}\Delta\martin^{*}(\path_{i-1})(x_i)
=\martin(\init)+\sum_{i=1}^{n}\Delta\martin(\path_{i-1})(x_i)
=\martin(\path_n)
,
\end{equation*}
where the second equality is a consequence of Equation~\eqref{eq:selstar1} and the fact that $\ell(\path_{i-1})=i-1$. Hence, using Equation~\eqref{eq:deflimsup}, we find that
\vspace{2pt}
\begin{equation*}
\limsup\martin^{*}(\path)
=
\limsup_{k\rightarrow+\infty}\martin^{*}(\path_k)
=\martin(\path_n),
\end{equation*}
as desired.
\end{proof}
\noindent
For gambles that are $n$-measurable, with $n\in\natswith$, both extensions coincide.
\begin{proposition}\label{prop:nWequalV}
For any situation $u$, any $n\in\natswith$ and any extended real-valued function $f$ on $\paths$ that is $n$-measurable, we have that $\lexpV(f)=\lexpW(f)$.
\end{proposition}
\begin{proof}
Consider any situation $u$, any $n\in\natswith$ and any extended real-valued function $f$ on $\paths$ that is $n$-measurable. Then due to Proposition~\ref{prop:WversusV}, we are left to prove that $\lexpW(f\vert u)\geq\lexpV(f\vert u)$. 
Let $m\coloneqq\max\{n,\ell(u)\}$. By Equation~\eqref{eq:Williams}, it is enough to show that, for all $\alpha\in\reals$ such that $\alpha<\lexpV(f\vert u)$:
\begin{equation}\label{eq:VequalWhulpvgl1} 
\exists\martin\in\martinsb
\colon\martin(u)\geq\alpha\text{~and~}
f(\path)\geq\martin(\path_m)\text{~for all $\path\in\cyl(u)$}. 
\end{equation}
So fix any $\alpha\in\reals$ such that $\alpha<\lexpV(f\vert u)$. By Equation~\eqref{eq:Ville}, there is some $\martin\in\martinsb$ such that
\vspace{-6pt}
\begin{equation}\label{eq:VequalWhulpvgl2}
\martin(u)\geq\alpha\text{~and~}
f(\omega)\geq\limsup\martin(\omega)\text{~for all $\path\in\cyl(u)$}.
\vspace{4pt}
\end{equation}
Consider any $\path\in\cyl(u)$ and let $s_{m}\coloneqq\path_{m}$. Since $\martin\in\martinsb$, we know that $\Delta\martin$ is almost-desirable, which implies that $\underline{Q}(\Delta\martin(s_{m})\vert s_{m})\geq0$ and therefore, because $\underline{Q}(\cdot\vert s_{m})$ is a lower expectation functional, it is not possible for $\Delta\martin(s_{m})$ to be uniformly negative. Hence, there is some $x_{m+1}^*\in\values_{m+1}$ such that $\Delta\martin(s_{m})(x_{m+1}^*)\geq0$. Now consider the situation $s_{m+1}\coloneqq s_m x_{m+1}^*$. Then similarly, there is an $\smash{x_{m+2}^*\in\values_{m+2}}$ such that $\smash{\Delta\martin(s_{m+1})(x_{m+2}^*)\geq0}$. Continuing in this way, we construct a path\footnote{In doing so, we implicitly adopt the Axiom of Dependent Choice.}
$\path^*=s_mx^*_{m+1}x^*_{m+2}\ldots$ for which, for all $i\geq m$, $\Delta\martin(\path^*_i)(x^*_{i+1})\geq0$. Therefore, and because of Equation~\eqref{eq:Fnclosedform}, we have for all $r\geq m$ that
\begin{align*}
\martin(\path^*_r)
=\martin(\init)+\sum_{i=1}^{r}\Delta\martin(\path^*_{i-1})(x^*_i)
\geq\martin(\init)
+\sum_{i=1}^{m}\Delta\martin(\path^*_{i-1})(x^*_i)
\geq\martin(\path^*_m)
\end{align*}
and, consequently, that
\vspace{2pt}
\begin{align}
\limsup\martin(\path^*)=\limsup_{r\rightarrow+\infty}\martin(\path^*_r)\geq\martin(\path^*_m).
\label{eq:VequalFinequality}
\end{align}
Since $m\geq\ell(u)$ and $\path^*_m=\path_m$, $\path\in\cyl(u)$ implies that $\path^*\in\cyl(u)$. Combined with Equations~\eqref{eq:VequalWhulpvgl2} and~\eqref{eq:VequalFinequality}, this implies that $f(\path^*)\geq\martin(\path^*_m)$ and therefore, since $f$ is $m$-measurable and $\path_m^*=\path_m$, that $f(\path)\geq\martin(\path_m)$. Since this holds for all $\path\in\cyl(u)$, and because $\martin(u)\geq\alpha$, we obtain Equation~\eqref{eq:VequalWhulpvgl1}.
\end{proof}

\section{Continuity with respect to non-decreasing sequences}
\label{sec:non-decreasing}

Since the Williams and Ville-Vovk-Shafer natural extension coincide on $n$-measurable functions, we can ask ourselves whether this holds for limits of $n$-measurable functions as well. The following result establishes that for non-decreasing sequences of $n$-measurable real-valued functions, this is indeed the case. Furthermore, both extensions are continuous with respect to the convergence of such a sequence.

\begin{theorem}\label{theo:non-decreasing}
Let $\{f_n\}_{n\in\natswith}$ be a non-decreasing sequence of $n$-measurable real-valued functions on $\paths$ and let $f\coloneqq\lim_{n\rightarrow+\infty}f_n$ be their pointwise limit. Then for any situation $u$, we have that
\begin{equation}\label{eq:non-decreasing}
\lim_{n\rightarrow+\infty}\lexpW(f_n\vert u)
=\lim_{n\rightarrow+\infty}\lexpV(f_n\vert u)
=\lexpV(f\vert u)=\lexpW(f\vert u).
\end{equation}
\end{theorem}
\begin{proof}
Since $\{f_n\}_{n\in\natswith}$, is non-decreasing, it holds for all $k,n\in\natswith$ such that $k\leq n$ that $f_k\leq f_n\leq f$ and therefore also, using Equation~\eqref{eq:Williams}, that $\lexpW(f_k\vert u)\leq\lexpW(f_n\vert u)\leq\lexpW(f\vert u)$. By combining this with Propositions~\ref{prop:WversusV} and~\ref{prop:nWequalV}, we find that
\begin{equation*}
\lim_{n\rightarrow+\infty}\lexpV(f_n\vert u)
=\lim_{n\rightarrow+\infty}\lexpW(f_n\vert u)
\leq\lexpW(f\vert u)
\leq\lexpV(f\vert u).
\end{equation*} 
Hence, we are left to show that $\lim_{n\rightarrow+\infty}\lexpV(f_n\vert u)
\geq\lexpV(f\vert u)$. We will do so by proving that, for all $\alpha\in\reals$ such that $\alpha<\lexpV(f\vert u)$, $\lim_{n\rightarrow+\infty}\lexpV(f_n\vert u)
\geq\alpha$.

So consider any $\alpha\in\reals$ such that $\alpha<\lexpV(f\vert u)$. Fix any $\epsilon>0$. Then due to Equation~\eqref{eq:Ville}, there is some $\martin\in\martinsb$ such that $\martin(u)\geq\alpha$ and, for all $\path\in\cyl(u)$, $f(\path)\geq\limsup\martin(\path)$. 
Fix any $\path\in\cyl(u)$. Since $\martin$ is bounded from above, we then know that $\limsup\martin(\path)\neq+\infty$. We consider two cases: $\limsup\martin(\path)\in\reals$ and $\limsup\martin(\path)=-\infty$. If $\limsup\martin(\path)\in\reals$, it follows from Equation~\eqref{eq:deflimsup} that there is some $m\in\natswith$ such that $\sup_{n\geq m}\martin(\path_n)\leq\limsup\martin(\path)+\epsilon$, and therefore, since $\lim_{n\to+\infty}f_n(\path)=f(\path)$ and $f(\path)\geq\limsup\martin(\path)$, we find that there is some $n\geq\max\{m,\ell(u)\}$ such that $f_n(\path)\geq\martin(\path_n)-2\epsilon$. If $\limsup\martin(\path)=-\infty$, it follows from Equation~\eqref{eq:deflimsup} that there is some $m\in\natswith$ such that $\sup_{n\geq m}\martin(\path_n)\leq f_0(\path)$, and therefore, for any $n\geq \max\{m,\ell(u)\}$, we find that $\martin(\path_n)\leq f_0(\path)\leq f_n(\path)$. Hence, we find that there is some $n\geq\ell(u)$ such that $f_n(\path)\geq\martin(\path_n)-2\epsilon$. Since $\path\in\cyl(u)$ is arbitrary, we conclude that for every $\path\in\cyl(u)$, there is some $n\geq\ell(u)$ such that $f_n(\path)-\martin(\path_n)\geq-2\epsilon$. For any $\path\in\cyl(u)$, let $n^*(\path)$ be the first $n\geq\ell(u)$ for which this is the case. For any $\path\in\paths\setminus\cyl(u)$, let $n^*(\path)=\ell(u)$.

For any $\path\in\cyl(u)$, any $\path'\in\cyl(\path_{n^*(\path)})$ and any $\ell(u)\leq n\leq n^*(\path)$, we now have that $f_n(\path)-\martin(\path_n)=f_n(\path')-\martin(\path'_n)$ because $\path_n=\path_n'$ and because $f_n$ is $n$-measurable. Hence, for any $\path\in\cyl(u)$ and any $\path'\in\cyl(\path_{n^*(\path)})$, we find that $n^*(\path')=n^*(\path)$. Since this is clearly also true for any $\path\in\paths\setminus\cyl(u)$ and $\path'\in\cyl(\path_{n^*(\path)})$, we find that
\begin{equation}\label{eq:nstarconstant}
n^*(\path')=n^*(\path)
\text{~for any $\path\in\paths$ and any $\path'\in\cyl(\path_{n^*(\path)})$}
\vspace{3pt}
\end{equation}


For all $n\in\natswith$, we now define $C_n\coloneqq\{\path\in\paths\colon n^*(\path)\geq n\}$. Then for all $n_1,n_2\in\natswith$ such that $n_1\leq n_2$, we have that $C_{n_1}\supseteq C_{n_2}$. Assume \emph{ex absurdo} that, for all $n\in\natswith$, $C_n\neq\emptyset$, implying that $\sup_{\path\in\paths}n^*(\path)=+\infty$.  Then we have that
\begin{equation*}
\sup_{\path\in\paths}n^*(\path)=\sup_{x\in\values_1}\sup_{\path\in\cyl(x)}n^*(\path)=+\infty
\end{equation*}
and therefore, since $\values_1$ is finite, there is clearly some $x_1^*\in\values_1$ for which $\sup_{\path\in\cyl(x_1^*)}n^*(\path)=+\infty$.
Similarly, we also have that
\begin{equation*}
\sup_{\path\in\cyl(x_{1}^*)}n^*(\path)=\sup_{x\in\values_{2}}\sup_{\path\in\cyl(x_{1}^*x)}n^*(\path)=+\infty.
\end{equation*}
Hence, since $\values_{2}$ is finite, there is some $\smash{x^*_{2}\in\values_{2}}$ for which $\sup_{\path\in\cyl(x_{1}^*x_{2}^*)}n^*(\path)=+\infty$. By continuing in this way, we construct a path\footnote{Again, we implicitly adopt the Axiom of Dependent Choice.} 
$\path^*=x_1^*\dots x_{n}^*\dots$ for which
\begin{equation*}
\sup_{\path\in\cyl(\path_n^*)}n^*(\path)
=+\infty
\text{ for all $n\in\natswith$}.
\end{equation*}
However, because of Equation~\eqref{eq:nstarconstant}, we also know that
\begin{equation*}
\sup_{\path\in\cyl(\path^*_{n^*(\path^*)})}n^*(\path)
=\sup_{\path\in\cyl(\path^*_{n^*(\path^*)})}n^*(\path^*)
=n^*(\path^*)
\neq+\infty.
\end{equation*}
Since this is a contradiction, we conclude that there is some $n^*\in\natswith$ for which $C_{n^*}=\emptyset$ and therefore also, for all $\path\in\paths$, $n^*(\path)<n^*$.

Next, let $\martin^{*}$ be the unique real proces defined by $\martin^{*}(\init)=\martin(\init)-2\epsilon$ and
\begin{equation}\label{eq:selstar}
\Delta\martin^{*}(s)
\coloneqq
\begin{cases}
\Delta\martin(s) & \text{if $n^*(\path)>\ell(s)$ for all $\path\in\cyl(s)$}\\
0 & \text{otherwise}
\end{cases}
\text{~~~for all situations $s$.}
\end{equation}
Then clearly, $\martin^{*}$ is a martingale.

Consider now any $\path\in\paths$ and $n\in\natswith$. We then have that $\path\in\cyl(\path_n)$ and $\ell(\path_n)=n$. Therefore, if $n\geq n^*(\path)$, we infer from Equation~\eqref{eq:selstar} that $\Delta\martin^{*}(\path_n)=0$. If $n<n^*(\path)$, then for any $\path'\in\cyl(\path_n)$, we have that $n^*(\path')>n=\ell(\path_n)$. Indeed, assume \emph{ex absurdo} that there is some $\path'\in\cyl(\path_n)$ for which $n^*(\path')\leq n$. Then $\path'_{n^*(\path')}=\path_{n^*(\path')}$, implying that $\path\in\cyl(\path'_{n^*(\path')})$ and therefore, by Equation~\eqref{eq:nstarconstant}, that $n^*(\path)=n^*(\path')\leq n$. This is a contradiction. 
Using Equation~\eqref{eq:selstar}, we find that $\Delta\martin^{*}(\path_n)=\Delta\martin(\path_n)$. Hence, in summary:
\begin{equation}\label{eq:selstarpath}
\Delta\martin^{*}(\path_n)
=
\begin{cases}
\Delta\martin(\path_n) & \text{if $n<n^*(\path)$}\\
0 & \text{otherwise}
\end{cases}
\text{~~~for all $\path\in\paths$ and $n\in\natswith$,}
\end{equation}
which implies that
\begin{equation}\label{eq:selstarpath:sum}
\martin^{*}(\path_n)
=
\begin{cases}
\martin(\path_n)-2\epsilon & \text{if $n<n^*(\path)$}\\
\martin(\path_{n^*(\path)})-2\epsilon & \text{otherwise}
\end{cases}
\text{~~~for all $\path\in\paths$ and $n\in\natswith$.}
\end{equation}
Therefore, and because $\martin$ is bounded above, we find that $\martin^{*}$ is also bounded above, which implies that $\martin^{*}\in\martinsb$.

Consider now any $\path\in\cyl(u)$. Then for all $n\geq n^*$, since $n^*(\path)< n^*$, Equation~\eqref{eq:selstarpath:sum} implies that 
$\martin^{*}(\path_n)=\martin(\path_{n^*(\path)})-2\epsilon$. 
Hence, we also have that
\begin{equation}\label{eq:limsupstar}
\limsup\martin^{*}(\path)=\limsup_{n\rightarrow+\infty}\martin^{*}(\path_n)=\martin(\path_{n^*(\path)})-2\epsilon\leq f_{n^*(\path)}(\path)
\leq f_{n^*}(\path),
\end{equation}
where the first inequality follows from the definition of $n^*(\path)$, and the second inequality is a consequence of the non-decreasing nature of the sequence $\{f_n\}_{n\in\natswith}$, and the fact that $n^*(\path)<n^*$.

Since Equation~\eqref{eq:limsupstar} holds for all $\path\in\cyl(u)$, and since $\martin^{*}\in\martinsb$, it follows from Equation~\eqref{eq:Ville} that $\lexpV(f_{n^*}\vert u)\geq\martin^{*}(u)=\martin(u)-2\epsilon\geq\alpha-2\epsilon$. Because $\{f_n\}_{n\in\natswith}$ is non-decreasing, we infer from Equation~\eqref{eq:Ville} that $\{\lexpV(f_n\vert u)\}_{n\in\natswith}$ is non-decreasing as well. Therefore, $\lexpV(f_{n^*}\vert u)\geq\alpha-2\epsilon$ implies that $\lim_{n\rightarrow+\infty}\lexpV(f_n\vert u)\geq\alpha-2\epsilon$. Since this holds for any $\epsilon>0$, we finally obtain that $\lim_{n\rightarrow+\infty}\lexpV(f_n\vert u)\geq\alpha$, as desired.
\end{proof}

\section{Continuity with respect to non-increasing sequences}
\label{sec:non-increasing}

The proof of the result in the previous section crucially hinges on the fact that $\{f_n\}_{n\in\natswith}$ is a non-decreasing sequence. For non-increasing sequences of $n$-measurable real-valued functions, the Williams natural extension is not as well-behaved, as it is not necessarily continuous with respect to the pointwise convergence of such sequences. For the Ville-Vovk-Shafer natural extension, however, results analogous to those in Theorem~\ref{theo:non-decreasing} can be obtained, provided that $\{f_n\}_{n\in\natswith}$ converges to a bounded function. Whether this assumption of boundedness is crucial in order for the result to hold, is an open question. 

\begin{theorem}\label{theo:non-increasing}
Let $\{f_n\}_{n\in\natswith}$ be a non-increasing sequence of $n$-measurable real-valued functions on $\paths$ and let $f\coloneqq\lim_{n\rightarrow+\infty}f_n$ be their pointwise limit. If $f$ is bounded, then for any situation $u$, we have that
\begin{equation}\label{eq:non-increasing}
\lim_{n\rightarrow+\infty}\lexpV(f_n\vert u)
=\lexpV(f\vert u).
\vspace{3pt}
\end{equation}
\end{theorem}
\begin{proof}
Since $f$ is bounded, there is some $B\in\reals$ such that $f\geq B$. Since the constant martingale $\martin\coloneqq B$ clearly belongs to $\martinsb$, it follows from Equation~\eqref{eq:Ville} that $\lexpV(B\vert u)\geq B$.
Since $\{f_n\}_{n\in\natswith}$ is non-increasing, it holds for all $k,n\in\natswith$ such that $k\leq n$ that $f_k\geq f_n\geq f\geq B$ and therefore also, using Equation~\eqref{eq:Ville}, that $\lexpV(f_k\vert u)\geq\lexpV(f_n\vert u)\geq\lexpV(f\vert u)\geq\lexpV(B\vert u)\geq B$. Therefore, we find that $\{\lexpV(f_n\vert u)\}_{n\in\natswith}$ is a non-increasing sequence and that
\vspace{-6pt}
\begin{equation*}
\lim_{n\rightarrow+\infty}\lexpV(f_n\vert u)
\geq\lexpV(f\vert u)\geq B,
\end{equation*}
implying that we are left to prove that $\lim_{n\rightarrow+\infty}\lexpV(f_n\vert u)
\leq\lexpV(f\vert u)$. Choose any $\alpha\in\reals$ such that $\alpha<\lim_{n\to+\infty}\lexpV(f_n\vert u)$.

Fix $\epsilon>0$, let $\underline{\beta}\coloneqq\min\{B,\alpha\}$ and let $\overline{\beta}$ be the value of the constant function $f_0$.

Fix any $n\in\natswith$. Then due to Equation~\eqref{eq:Ville}, since $\alpha<\lexpV(f_n\vert u)$, there is a submartingale $\martin'_n\in\martinsb$ such that $\martin'_n(u)\geq\alpha$ and, for all $\path\in\cyl(u)$, $f_n(\path)\geq\limsup\martin'_n(\path)$. For any situation $s$, we now let
\begin{equation}\label{eq:martinn}
\martin_n(s)\coloneqq
\begin{cases}
\martin'_n(u)
&\text{if $u\not\sprecedes s$}\\
\martin'_n(s)
&\text{if $u\sprecedes s$ and $\martin'_n(s')>B$ for all $u\sprecedes s'\precedes s$}\\
B
&\text{otherwise.}
\end{cases}
\vspace{10pt}
\end{equation}
Since $\martin'_n(u)\geq\alpha$, we then find that $\martin_n(u)\geq\alpha$ and that $\martin_n\geq\underline{\beta}$. Furthermore, for all $\path\in\cyl(u)$, we have that 
\vspace{3pt}
\begin{equation*}
\limsup\martin_n(\path)
\leq
\max\{\limsup\martin'_n(\path),B\}
\leq
\max\{ f_n(\path),B\}
\leq
f_n(\path),
\vspace{3pt}
\end{equation*}
where the last inequality holds because $f_n(\path)\geq f(\path)\geq B$. Next, we show that $\martin_n\in\martinsb$. Fix any situation $s$. We consider three cases. If $u\sprecedes s$ and $\martin'_n(s')>B$ for all $u\sprecedes s'\precedes s$, then $\martin_n(s)=\martin'_n(s)$ and $\martin_n(sx)\geq\martin'_n(sx)$ for all $x\in\values_{\ell(s)+1}$, and therefore $\Delta\martin_n(s)\geq\Delta\martin'_n(s)$. Hence, since $\underline{Q}(\cdot\vert s)$ is a lower expectation functional, we find that
\vspace{3pt}
\begin{align*}
\underline{Q}(\Delta\martin_n(s)\vert s)
&\geq
\underline{Q}(\Delta\martin'_n(s)\vert s)
+
\underline{Q}(\Delta\martin_n(s)-\Delta\martin'_n(s)\vert s)\\
&\geq
\underline{Q}(\Delta\martin'_n(s)\vert s)
+
\inf\big(\Delta\martin_n(s)-\Delta\martin'_n(s)\big)
\geq\underline{Q}(\Delta\martin'_n(s)\vert s)\geq0,\\[-10pt]
\end{align*}
where the first inequality follows from the superlinearity of $\underline{Q}(\cdot\vert s)$, the second inequality follows from the fact that $\underline{Q}(\cdot\vert s)$ dominates the infimum, the third inequality follows because $\Delta\martin_n(s)\geq\Delta\martin'_n(s)$, and the last inequality follows because $\martin'$ is a submartingale.
If $s=u$, then $\martin_n(s)=\martin'_n(s)$ and $\martin_n(sx)\geq\martin'_n(sx)$ for all $x\in\values_{\ell(s)+1}$, and therefore $\Delta\martin_n(s)\geq\Delta\martin'_n(s)$. Hence, as before, we find that $\underline{Q}(\Delta\martin_n(s)\vert s)\geq0$. In all other cases, we have that $\martin_n(s)=\martin_n(sx)$ for all $x\in\values_{\ell(s)+1}$, and therefore $\Delta\martin_n(s)=0$. Since $\underline{Q}(\cdot\vert s)$ dominates the infimum, this implies that $\underline{Q}(\Delta\martin_n(s)\vert s)\geq0$. Hence, in all cases, we find that $\underline{Q}(\Delta\martin_n(s)\vert s)\geq0$. Since this is true for every situation $s$, it follows that $\martin_n\in\martins$. Furthermore, since $\martin'_n\in\martinsb$, there is some $B'\in\reals$ such that $\martin'_n\leq B'$. Because of Equation~\eqref{eq:martinn}, this implies that $\martin_n\leq\max\{B,B'\}$, which in turn implies that $\martin_n$ is bounded above. Therefore, since $\martin_n\in\martins$, we find that $\martin_n\in\martinsb$.

Hence, in conclusion, we find that for every $n\in\natswith$, there is some $\martin_n\in\martinsb$ such that $\martin_n(u)\geq\alpha$, $\martin_n\geq\underline{\beta}$ and, for all $\path\in\cyl(u)$, $\limsup\martin_n(\path)\leq f_n(\path)$.

Consider now any $k,n\in\natswith$ such that $k\leq n$ and any $\path\in\cyl(u)$.
Let $s_k\coloneqq\path_k$. Since $\martin_n\in\martinsb$, 
we can use an argument similar to the one used in the proof of Proposition~\ref{prop:nWequalV} to construct a path $\path^*=s_kx^*_{k+1}x^*_{k+2}\ldots x^*_n\ldots$ for which, for all $m\geq k$, $\martin_n(\path^*_m)\geq\martin_n(\path_k)$
and therefore also  
\begin{align}
\martin_n(\path_k)
\leq
\limsup_{m\rightarrow+\infty}\martin_n(\path^*_m)
=\limsup\martin_n(\path^*)
\leq f_n(\path^*)
\leq f_k(\path^*)
=f_k(\path),
\label{eq:knpath}
\end{align}
where the last equality holds because $f_k$ is $k$-measurable and because $\path^*_k=s_k=\path_k$.

Next, we consider an extended real-valued process $\proc$, defined by
\vspace{2pt}
\begin{equation}\label{eq:defF}
\proc(s)\coloneqq\limsup_{n\rightarrow+\infty}\martin_n(s)\geq\underline{\beta}
\text{ for all situations $s$.}
\vspace{2pt}
\end{equation}
For any $\path\in\cyl(u)$ and any $k\in\natswith$, we then have that
\begin{equation}\label{eq:Fpathk}
\proc(\path_k)
=\limsup_{n\rightarrow+\infty}\martin_n(\path_k)
\leq f_k(\path)\leq f_0(\path)=\overline{\beta},
\end{equation}
where the first inequality follows because Equation~\eqref{eq:knpath} holds for every $n\geq k$. Using Equation~\eqref{eq:Fpathk}, we now find that, for all $\path\in\cyl(u)$,
\begin{equation}\label{eq:cyluinequality}
\limsup\proc(\path)
=\limsup_{k\rightarrow+\infty}\proc(\path_k)
\leq
\limsup_{k\rightarrow+\infty}f_k(\path)
=f(\path),
\end{equation}
where the last equality holds because $\{f_k\}_{k\in\natswith}$ converges pointwise to $f$. 

We now construct an almost-desirable selection $\sel\in\sels$.
Consider any situation $s$. We first consider the case $u\precedes s$. It then follows from Equations~\eqref{eq:defF} and~\eqref{eq:Fpathk} that $\proc(s)\in\reals$ and, for all $x\in\values_{\ell(s)+1}$, that $\proc(sx)\in\reals$. Hence, for any $x\in\values_{\ell(s)+1}$, since $\proc(sx)=\limsup_{n\rightarrow+\infty}\martin_n(sx)$, we find that there is some $n_x(s)\in\natswith$ for which, for all $n\geq n_x(s)$, $\proc(sx)\geq\martin_n(sx)-\nicefrac{\nu(s)}{2}$, with $\nu(s)\coloneqq\epsilon2^{-(\ell(s)+1)}>0$. Because $\values_{\ell(s)+1}$ is finite, we can define $n_{\max}(s)\coloneqq\max\{n_x(s)\colon x\in\values_{\ell(s)+1}\}$. Since $\proc(s)=\limsup_{n\rightarrow+\infty}\martin_n(s)$, there is some $n^*(s)\geq n_{\max}(s)$ such that $\martin_{n^*(s)}(s)\geq\proc(s)-\nicefrac{\nu(s)}{2}$. If we now define $\sel(s)\coloneqq\Delta\martin_{n^*(s)}(s)$, then for all $x\in\values_{\ell(s)+1}$, since $n^*(s)\geq n_{\max}(s)\geq n_{x}(s)$, we find that
\begin{align*}
\Delta\proc(s)(x)
=\proc(sx)-\proc(s)
&\geq
\Big(\martin_{n^*(s)}(sx)-\frac{\nu(s)}{2}\Big)+\Big(-\martin_{n^*(s)}(s)-\frac{\nu(s)}{2}\Big)
\\
&=
\martin_{n^*(s)}(sx)-\martin_{n^*(s)}(s)-\nu(s)
=
\sel(s)(x)-\nu(s).\\[-10pt]
\end{align*}
Furthermore, since $\martin_{n^*(s)}\in\martinsb$, we also know that $\underline{Q}(\sel(s)\vert s)\geq0$. We next consider the case $u\not\precedes s$. In that case, we let $\sel(s)\coloneqq0$. Here too, since $\underline{Q}(\cdot\vert s)$ is a lower expectation functional, we have that $\underline{Q}(\sel(s)\vert s)\geq0$. Therefore, it follows that $\sel$ indeed belongs to $\sels$. Now let $\martin\coloneqq\proc(u)-\epsilon+\proc^{\sel}$. Since we know from Equations~\eqref{eq:defF} and~\eqref{eq:Fpathk} that $\proc(u)\in\reals$, and since $\sel\in\sels$, we know that $\martin\in\martins$.

Consider now any situation $s=x_1\dots x_n$. If $u\not\sprecedes s$, we find that
\begin{equation}
\martin(s)=\proc(u)-\epsilon+\sum_{i=0}^{n-1}\sel(x_1\dots x_i)(x_{i+1})=\proc(u)-\epsilon\leq\overline{\beta},\label{eq:unotsprecedess}
\end{equation}
using Equation~\eqref{eq:Fpathk} for the last inequality. If $u\sprecedes s$, then
\begin{align}
\martin(s)&=\proc(u)-\epsilon+\sum_{i=0}^{n-1}\sel(x_1\dots x_i)(x_{i+1})\notag\\
&=\proc(u)-\epsilon+\sum_{i=\ell(u)}^{n-1}\sel(x_1\dots x_i)(x_{i+1})\notag\\
&\leq\proc(u)-\epsilon+\sum_{i=\ell(u)}^{n-1}\Delta\proc(x_1\dots x_i)(x_{i+1})+\sum_{i=\ell(u)}^{n-1}\nu(x_1\dots x_i)\notag\\
&=\proc(s)-\epsilon+\sum_{i=\ell(u)}^{n-1}\nu(x_1\dots x_i)\notag\\
&=\proc(s)-\epsilon+\sum_{i=\ell(u)}^{n-1}\epsilon2^{-(i+1)}
\leq\proc(s)-\epsilon+\sum_{i\in\natswith}\epsilon2^{-(i+1)}
=\proc(s)
\leq\overline{\beta},\label{eq:betaepsilon}\\[-10pt]\notag
\end{align}
again using Equation~\eqref{eq:Fpathk} for the last inequality. Hence, for every situation $s$, we find that $\martin(s)\leq\overline{\beta}$. This implies that $\martin$ is bounded above, and therefore, that $\martin\in\martinsb$.

Consider now any path $\path=x_1\ldots x_n\ldots\in\cyl(u)$. Then for every $m>\ell(u)$, since $u\sprecedes\path_m$, we know from Equation~\eqref{eq:betaepsilon} that $\martin(\path_m)\leq\proc(\path_m)$. Therefore, we find that
\begin{equation*}
\limsup\martin(\path)
=\limsup_{m\rightarrow+\infty}\martin(\path_m)
\leq\limsup_{m\rightarrow+\infty}\proc(\path_m)
=\limsup\proc(\path)\leq f(\path),
\end{equation*}
using Equation~\eqref{eq:cyluinequality} for the last inequality.
Since this holds for all $\path\in\cyl(u)$, and because $\martin\in\martinsb$, it follows from Equation~\eqref{eq:Ville} that $\lexpV(f\vert u)\geq\martin(u)$, which implies that
\begin{equation*}
\lexpV(f\vert u)\geq\martin(u)=\proc(u)-\epsilon=\limsup_{n\to+\infty}\martin_n(u)-\epsilon\geq\alpha-\epsilon,
\end{equation*}
where the first equality follows from Equation~\eqref{eq:unotsprecedess}, where the second equality follows from Equation~\eqref{eq:defF}, and where the last inequality follows from the fact that, for all $n\in\natswith$, $\martin_n(u)\geq\alpha$. Hence, we find that $\lexpV(f\vert u)\geq\alpha-\epsilon$. Since $\epsilon>0$ is arbitrary, this implies that $\lexpV(f\vert u)\geq\alpha$. Since this is true for any $\alpha\in\reals$ such that $\alpha<\lim_{n\to+\infty}\lexpV(f_n\vert u)$, we conclude that $\lexpV(f\vert u)\geq\lim_{n\to+\infty}\lexpV(f_n\vert u)$.
\end{proof}

\section*{Acknowledgements}
The results in this note have benefited from discussions with Enrique Miranda, Matthias C. M. Troffaes and Arthur Van Camp. Special thanks to Matthias C. M. Troffaes for proofreading the initial part of an earlier version.

\bibliographystyle{plain} 

\end{document}